\documentclass[11pt]{amsart}
\usepackage{graphicx}
\usepackage{subfig}
\usepackage{tabularx}
\usepackage{wasysym,mathtools}

\usepackage{mathrsfs}     
\usepackage{helvet}         
\usepackage{courier}        
\usepackage{type1cm}      
\usepackage{color}

\usepackage{geometry,calc,color}
\usepackage{amsmath,amssymb}

\usepackage{enumerate}

\newtheorem{theorem}{Theorem}[section]
\newtheorem{corollary}[theorem]{Corollary}
\newtheorem{lemma}[theorem]{Lemma}
\newtheorem{proposition}[theorem]{Proposition}
\theoremstyle{definition}

\newtheorem{problem}[theorem]{Problem}
\newtheorem{assumption}[theorem]{Assumption}
\newtheorem{comment}[theorem]{Comment}
\theoremstyle{remark}
\newtheorem{remark}[theorem]{Remark}

\numberwithin{equation}{section}

\newcommand{\nref}[1]{(\ref{#1})}

\newcommand{\Th}{\mathcal{T}_h}
\newcommand{\Thk}{\mathcal{T}_h^\ell}
\newcommand{\Eh}{\mathcal{E}_h}
\newcommand{\EH}{\mathcal{E}_H}
\newcommand{\EK}{\mathcal{E}^K}

\newcommand{\Yl}{\mathcal{Y}^\ell}
\newcommand{\Xl}{\mathcal{X}^\ell}

\newcommand{\cX}{\mathcal{X}}
\newcommand{\cY}{\mathcal{Y}}

\newcommand{\Ni}{\mathcal{N}_i}

\newcommand{\XE}{\mathcal{X}_E}
\newcommand{\YE}{\mathcal{Y}_E}
\newcommand{\El}{\mathcal{E}_\ell}

\newcommand{\hW}{\widehat W_h}
\newcommand{\hR}{\widehat R}
\newcommand{\tW}{\widetilde W_h}
\newcommand{\tR}{\widetilde R}
\newcommand{\hSop}{\mathcal{\widehat S}}
\newcommand{\tSop}{\mathcal{\widetilde S}}

\newcommand{\tSmat}{{\widetilde S}}
\newcommand{\Rop}{\mathcal{R}}
\newcommand{\Bop}{\mathcal{B}}

\newcommand{\Lifthl}{\mathcal{L}_h^\ell}
\newcommand{\Lifth}{\mathcal{L}_h}

\renewcommand{\vec}[1]{\mathbf{#1}}

\newcommand{\B}{\mathcal{B}}

\newcommand{\Svem}{S_a^K}

\newcommand{\ED}{\mathcal E_D}

\newcommand{\BDDC}{\mathcal{M}_{BDDC}}
\newcommand{\FETI}{\mathcal{M}_{FETI}}
\newcommand{\IdtW}{\mathbf{1}_{\tW}}

\newcommand{\oml}{\Omega^\ell}
\newcommand{\omk}{\Omega^m}

\newcommand{\normWh}[1]{\| #1 \|_{1/2,*}}
\newcommand{\snormWh}[1]{| #1 |_{1/2,*}}

\newcommand{\Pinabla}{\Pi^\nabla_K}

\newcommand{\tN}{\widetilde{N}}
\newcommand{\hN}{\widehat{N}}
\newcommand{\Nl}{{N_\ell}}
\newcommand{\ND}{{N_\Delta}}
\newcommand{\NP}{{N_\Pi}}

\newcommand{\DD}{{\Delta\Delta}}
\newcommand{\PP}{{\Pi\Pi}}
\newcommand{\DP}{{\Delta\Pi}}

\renewcommand{\Re}{\mathbb{R}}

\newcommand{\Ndofs}{N_{dof}}

\newcommand{\BD}{{B_D}}

\newcommand{\DB}{D^\flat}

\newcommand{\tS}{\widetilde S}
\newcommand{\hS}{\widehat S}

\newcommand{\Pscottzhang}{\Pi_{SZ}}

\begin{document}
	
	\title[BDDC and FETI-DP for VEM]{BDDC and FETI-DP for the Virtual Element Method}%

	\author[S. Bertoluzza]{Silvia Bertoluzza}
	\address{IMATI ``E. Magenes'', CNR, Pavia (Italy)}%
	\email{silvia.bertoluzza@imati.cnr.it}%
	
	\author[M. Pennacchio]{Micol Pennacchio}
	\address{IMATI ``E. Magenes'', CNR, Pavia (Italy)}%
	\email{micol.pennacchio@imati.cnr.it}%
	
	\author[D. Prada]{Daniele Prada}
	\address{IMATI ``E. Magenes'', CNR, Pavia (Italy)}%
	\email{daniele.prada@imati.cnr.it}%

	\thanks{{This paper has been realized in the framework of ERC Project CHANGE, which has received funding from the European Research Council (ERC) under the European Union’s Horizon 2020 research and innovation programme (grant agreement No 694515).
	\newline
This is a post-peer-review, pre-copyedit version of an article published in Calcolo,  54:1565-1593 (2017). The final authenticated version is avaliable online at: https://doi.org/10.1007/s10092-017-0242-3.
  }}%
	
		\maketitle
		
	\begin{abstract} 
		We build and analyze Balancing Domain Decomposition by Constraint (BDDC) and Finite Element Tearing and Interconnecting Dual Primal (FETI-DP)  preconditioners  for  elliptic problems discretized by the virtual element method (VEM).  
		We prove polylogarithmic condition number bounds, independent of the number of subdomains, the mesh size, and jumps in the diffusion coefficients. 
		Numerical experiments confirm the theory. 
	\end{abstract}

	
	\section{Introduction}
	\label{intro}
	
	The ever-increasing interest in  methods based on polygonal and polyhedral meshes for the numerical solution of PDEs  stems  from 
	the high flexibility that polytopic grids allow in the treatment  of complex geometries, which presently turns out to be, in many applications in computational engineering and scientific computing, as crucial a task as the construction and numerical solution of the discretized equations. 
	Examples of methods  where the discretization is based on arbitrarily shaped polytopic meshes include: 
	Mimetic Finite Differences \cite{book_mimetic,brezzi_mimetic}, Discontinuous Galerkin-Finite Element Method (DG-FEM) \cite{Rev_DG,Cangiani_hpDGVEM}, Hybridizable and Hybrid High-Order Methods
	\cite{Cockburn_LDG,Ern_LDG}, Weak Galerkin Method  \cite{Weak_FEM}, BEM-based FEM \cite{BEM_FEM}  and Polygonal FEM \cite{pol_FEM}. 
	Among such methods, the virtual element method (VEM) \cite{basicVEM} is a quite recent  discretization framework which can be viewed as an extension of the Finite Element Methods (FEM). 
	The  main idea of VEM is to consider local approximation spaces including polynomial functions, but to avoid the explicit construction and integration of the associated shape functions, whence the name {\em virtual}.
	An implicit knowledge of the local shape functions allows the evaluation of the operators and matrices needed in the method. 
	Its implementation is described in \cite{hitchVEM} and the $p$ and $hp$ versions of the method are discussed and analyzed in \cite{antonietti_p_VEM,beirao_hp,beirao_hp_exponential}.
	Despite its recent introduction, VEM has already been applied and extended to study a wide variety of different model problems. 
	Within the  VEM literature we recall applications to: parabolic problems \cite{beirao_parab}, Cahn-Hilliard, Stokes, Navier-Stokes and Helmoltz equations \cite{Antonietti_VEM_Stokes,Antonietti_VEM_Cahn,beirao_stokes,beirao_Navier_Stokes,perugia_Helmholtz}, linear and nonlinear elasticity problems \cite{beirao_elastic,beirao_linear_elasticity,VEM_3D_elasticity}, general elliptic problems in mixed form \cite{VEM_mixed}, fracture networks \cite{VEM_discrete_fracture}, Laplace-Beltrami equation \cite{VEM_Laplace_Beltrami}.

	To ensure that the Virtual Element Method can achieve its full potential, it is however necessary to deal with the efficient solution of the associated linear system of equations, and, in particular, to provide good preconditioners. By proposing suitable choices for the degrees of freedom, the first works in this direction aimed at dealing with the increase of condition number of the stiffness matrix resulting from the degradation of the quality of the geometry \cite{beirao_hp,antonietti_p_VEM,Mascotto_illcond} and/or to the increase in the polynomial order  of the method \cite{berrone_borio_17}.
	In this paper we rather focus on the increase in the condition number resulting from decreasing the mesh-size. In view of a possible parallel implementation of the method, we choose  to consider a Domain Decomposition approach \cite{Toselli.Widlund}.
	In particular, we focus on the two preconditioning techniques which are, nowadays, considered as the most efficient: 
	Balancing Domain Decomposition by Constraint (BDDC) \cite{Dohrmann.03} and Dual-Primal Finite Element Tearing and Interconnecting (FETI-DP) \cite{Farhat:partI}. These are non overlapping domain decomposition methods, aiming at preconditioning the Schur complement with respect to the unknowns on the skeleton of the subdomain partition.
	Acting on dual problems -- the BDDC method defines a preconditioner for the Schur complement of the linear matrix stemming from discretized PDE, whereas FETI-DP  reformulates the problem as a constrained optimization problem  and solves it by iterating on the set of Lagrange multipliers representing the fluxes across the interface between the non overlapping subdomains -- the two methods  have been proven to be spectrally equivalent \cite{brennersung,MandelSousedik,Mandel_algebraic}.
	While both the BDDC and FETI-DP methods have been already extensively studied in the context of many different discretization methods --  spectral elements  \cite{Pavarino2007.BDDC.FETIDP,klawonn2008spectral}, mortar discretizations \cite{KIM.mortar.BDDC,KIM.Tu.mortar.BDDC,KIM.Dryja.Widlund.08,KIM.mortar.FETIDP}, Discontinuous Galerkin methods \cite{dryja2007bddc,canuto2014bddc}, NURBS discretizations in isogeometric analysis \cite{beirao_BDDC_iso,beirao_BDDC_iso_deluxe} --
	to the best of our knowledge, these preconditioners have not yet been considered for VEM methods.

	Here we prove that properties of scalability, quasi-optimality and independence on the discontinuities of the elliptic operator coefficients across subdomain interfaces still hold when dealing with VEM. 
	More specifically,
	for both preconditioners, we show that the condition number of the preconditioned matrix is bounded by a constant times
	$$
	(1+\log(Hk^2/h))^2,
	$$
	where $H$, $h$ and $k$ are the subdomain mesh-size, the fine mesh-size and the polynomial order respectively, see Corollary \ref{cond_bound}. In order to prove such a result we show that the validity of several inequalities (see, e.g. Lemma \ref{lem:technical} 
	in the following) is independent of the properties of the discretization within the subdomains and only relies on the properties of the trace of the discretization on the interface. This allows us to avoid handling discrete functions which are only  implicitly defined.

	The  
	paper is organized as follows. The basic notation,
	functional setting and the  description of the Virtual Element Method are given
	in  Section~\ref{sec:VEM}. 
	The dual-primal preconditioners are introduced and analyzed in Section~\ref{sec:dualprimal}, whereas their 
	algebraic forms  are  presented 
	in Section~\ref{algebraic}.
	Numerical experiments that validate the theory are presented in Section~\ref{sec:expes}.

	\section{The virtual element method (VEM)}\label{sec:VEM}
	
	For $\hat \Omega$ (resp. $\hat \Gamma$) denoting any two-dimensional (resp. one-dimensional) domain with diameter (resp. length) $\hat H$, we introduce the usual (properly scaled) norms and semi-norms for the functional spaces that we will need to use in the following: 
	\begin{gather*}
		\| w \|_{L^2(\hat{\Omega})}^2 = \hat H^{-2} \int_{\hat\Omega} | w(x) |^2\,dx, \qquad \| w \|_{L^2(\hat{\Gamma})}^2 =\hat H^{-1} \int_{\hat\Gamma} | w(\tau) |^2\,d\tau,\\
		| w |_{H^1(\hat{\Omega})}^2 = \int_{\hat\Omega} | \nabla w(x) |^2\,dx, \qquad | w(\tau) |_{H^1(\hat{\Gamma})}^2 = \hat H \int_{\Gamma} | w'(\tau) |^2\,d\tau,\\
		| w |^2_{H^s(\hat{\Gamma})} = \hat H^{2s-1} \int_{\hat\Gamma}\,d\sigma \int_{\hat{\Gamma}}\,d\tau \frac {|w(\sigma) - w(\tau)|^2}{|\sigma-\tau|^{2s+1}  }, \ 0<s< 1,\\
		\| w \|_{H^1(\hat{\Omega})}^2 = \| w \|_{L^2(\hat{\Omega})}^2 + | w |_{H^1(\hat{\Omega})}^2, \quad \| w \|_{H^s(\hat{\Gamma})}^2 = \| w \|_{L^2(\hat{\Gamma})}^2 + | w |_{H^s(\hat{\Gamma})}^2,\ 0<s\leq 1.
	\end{gather*}
	
	\
	
	Let us start by recalling the definition and the main properties of the  Virtual Element Method \cite{basicVEM}. To fix the ideas we focus on the following elliptic model  problem: 
	\[
	-\nabla \cdot(\rho \nabla u) = f \ \text{ in } \Omega, \qquad u = 0 \ \text{ on } \partial \Omega,
	\]
	with $f \in L^2(\Omega)$,
	where $\Omega \subset \mathbb{R}^2$ is a polygonal domain. We assume that the coefficient $\rho$ is a scalar such that for almost all $x \in \Omega$,  $\alpha \leq \rho(x) \leq M$ for two constants $M \geq \alpha >0$. The variational formulation of such an equation reads
	\begin{equation}\label{variational}
		\left\{
		\begin{array}{l}
			\text{find $u \in V:=H^1_0(\Omega)$ such that}\\[1mm]
			a(u,v) = ( f , v ) \ \forall v \in V 
		\end{array}\right.
	\end{equation}
	with
	\[
	a(u,v) = \int_{\Omega}\rho(x)\nabla u(x)\cdot \nabla v(x)\,dx, \qquad (f,v) = \int_{\Omega}f(x) v(x)\,dx.
	\]
	
	\
	
	We consider a family $\{ \Th \}_h$ of tessellations of $\Omega$ into a finite number of  simple polygons $K$, and we let $\Eh$ be the set of edges $e$ of $\Th$. We make the following assumptions on the tessellation
	\begin{assumption}
		There exists constants $\gamma_0,\gamma_1 > 0$ such that for all tessellations $\Th$:
		\begin{enumerate}[(i)]
			\item  each element $K \in \Th$ is star-shaped with respect to a ball of radius $\geq \gamma_0 h_K$, where $h_K$ is the diameter of $K$;
			\item  for each element $K$  in $\Th$ the distance between any two vertices of $K$ is $\geq \gamma_1 h_K$.
		\end{enumerate}
	\end{assumption}
	
	\
	
	For simplicity, we also make the following assumption
	\begin{assumption}\label{ass:uniform} The tessellation $\Th$ is quasi-uniform, that  is there exist positive constants $c_0$, $c_1$  such that for any two elements $K$ and $K'$ in $\mathcal{T}_h$ we have $\alpha_0  \leq h_K / h_{K'}\leq \alpha_1 $. 
	\end{assumption}
	
	\begin{assumption}\label{ass:roconst} For all $K$ there exists a constant $\rho_K$ such that the coefficient $\rho$ verifies $\rho = \rho_K$ on $K$.
	\end{assumption}
	
	\begin{remark}
		Assumption \ref{ass:roconst} ensures that for the computable projection $\Pi^\nabla_K$ the splitting \nref{1} holds. Remark, however, that the virtual element method can also be defined for the case of coefficients varying within the elements of the tessellation \cite{beirao_var_coef}, and that the theoretical analysis developed here, which essentially relies on the bound \nref{elemequiv}, can be extended to such a case (see Section \ref{sec:extension}).
	\end{remark}

	\
	
	The Virtual Element discretization space is first defined element by element starting from the edges of the tessellation. More precisely, for each  polygon $K \in \Th$ we define the space $\mathbb{B}_k(\partial K)$ as
	\begin{gather*}
		\mathbb{B}_k(\partial K) = \{ v \in C^0(\partial K): v|_{e} \in \mathbb{P}_k\ \forall e \in \EK \},
	\end{gather*}
	where $\mathbb{P}_k$ denotes the set of polynomials of degree less than or equal to $k$, and where $\EK$ denotes the set of edges of the polygon $K$. It is not difficult to see that $\mathbb{B}_k(\partial K)$ is a linear space of dimension $kn^K$ where $n^K$ is the number of vertices of the polygon $K$.
	Letting \begin{equation}\label{deflocalspace}
		V^{K,k} =   \{ v \in H^1(K):\ v|_{\partial K} \in \mathbb{B}_k(\partial K),\ \Delta v \in \mathbb{P}_{k-2}(K) \}
	\end{equation}
	(with $\mathbb{P}_{-1}=\{0\}$), the discrete space $V_h$ is defined as
	\begin{align*}
		V_h &= \{ v \in V: w|_{K} \in V^{K,k}\ \forall K \in \Th \} =\\
		& = \{ v \in V:\forall K \in \Th\  w|_{\partial K} \in \mathbb{B}_k(\partial K),\ \Delta v|_K \in  \mathbb{P}_{k-2}(K) \}.
	\end{align*}
	
	\
	
	A function $v_h \in V_h$ is uniquely determined by the following degrees of freedom 
	\begin{itemize}
		\item the values of $v_h$ at the vertices of the tessellation;
		\item for each edge $e$, the values of $v_h$ at the $k-1$ internal points of the $k+1$-points Gauss-Lobatto quadrature rule on $e$;
		\item for each element $K$, the moments up to order $k-2$ of $v_h$ in $K$. 
	\end{itemize} 
	The degrees of freedom are then naturally split as {\em boundary degrees of freedom} (the first two sets) and {\em interior degrees of freedom}. Boundary degrees of freedom are nodal, and we let $\Upsilon$ denote the corresponding set of nodes, which includes the vertices of the tessellation  as well as the union over all edges of the $k-1$ internal nodes of the $k+1$-points  Gauss-Lobatto quadrature rule.

	Using a Galerkin approach, the Virtual Element Method stems from replacing the bilinear form $a$ (which, due to the implicit definition of the discrete space, is not directly computable on discrete functions in terms of the degrees of freedom) with a suitable approximate bilinear form. More precisely, setting, 
	for each $K \in \Th$
	\[
	a^K(u,v)=\int_K \rho \nabla u\cdot \nabla v,
	\]
	the key observation is that, given any $v \in V^{K,k}$ and any $p \in \mathbb{P}_k(K)$ we can compute $a^K(p,v)$ by using Green's formula (recall that $\rho=\rho_K$ constant in $K$)
	\[
	a^K(p,v) = - \rho_K\int_K v \Delta p + \rho_K \int_{\partial K} v \frac {\partial p}{\partial n}.
	\] 
	Since $v$ is a piecewise polynomial on $\partial K$,  $p$  a polynomial of order $k$, and $\Delta p$ a polynomial of order $k-2$, the right hand side can be computed directly from the degrees of freedom of $v_h$. This allows to define the ``element by element'' computable projection
	operator $\Pinabla : V^{K,k} \longrightarrow \mathbb{P}_k(K)$  as
	\[
	a^K(\Pinabla  u, q ) = a^K(u,q) \quad \forall q \in \mathbb{P}_k(K).
	\]
	The last equation defines $\Pinabla u$ only up to a constant; this is fixed by prescribing
	\begin{align*}
		&\sum_{i=1}^n \Pinabla u(V_i) =  \sum_{i=1}^n u(V_i),\quad V_i \text{ vertices of }K, & &\text{for }k = 1,\\
		&\int_K\Pinabla u = \int_K u, & &\text{for }k \geq 2.
	\end{align*}

	Clearly
	we have
	\begin{equation}\label{1}
		a^{K}(u,v) = a^K(\Pinabla  u,\Pinabla  v) + a^K(u - \Pinabla  u, v - \Pinabla  v).
	\end{equation}
	
	\

	The virtual element method stems 
	from replacing the second term of the sum on the right hand side (which cannot be computed exactly), with an ``equivalent''  term, where the bilinear for $a^K$ is substituted by a computable symmetric bilinear $\Svem$, resulting in defining
	\[
	a^{K}_h(u,v) = a^K(\Pinabla  u,\Pinabla  v) + \Svem(u - \Pinabla  u, v - \Pinabla  v) .
	\]
	Different choices are possible for the bilinear form $\Svem$ (see \cite{beirao_stab}), the essential requirement being that it satisfies
	\begin{equation}\label{2}
		c_0  a^K(v,v) \leq  \Svem(v,v) \leq c_1 a^K(v,v),\quad \forall v \in V^{K,k}\ \text{ with } \Pinabla  v=0,
	\end{equation}
	for two positive constants $c_0$ and $c_1$, so that we have 
	\begin{equation}\label{elemequiv}
		(1+c_0)   a^K(v,v) \leq   a^K_h(v,v) \leq (1+c_1) a^K(v,v) ,\quad \forall v \in V^{K,k}.
	\end{equation}
	In the numerical tests performed in Section \ref{sec:expes} we made the standard choice of defining $\Svem$ in terms of the vectors of local degrees of freedom as the properly scaled euclidean scalar product. 
	
	\
	
	Finally, we let $a_h : V_h \times V_h \to \mathbb{R}$ be defined by
	\[
	a_h(u_h,v_h) = \sum_K a_h^K(u_h,v_h),
	\]
	and we consider  the following discrete problems of the form
	\begin{problem}\label{discrete_full} Find $u_h \in V_h$ such that
		\[  a_h(u_h,v_h) = f_h(v_h) \qquad  \forall v_h \in V_h. \]
	\end{problem}
	
	\
	
	For the study of the convergence, stability and robustness properties of the method we refer to~\cite{basicVEM,beirao_var_coef}.

	\section{Domain decomposition for the Virtual Element Method}\label{sec:dualprimal}
	We assume that the set $\Th$ can be split as $\Th = \cup_\ell \Thk$, in such a way that $\Omega$ may be written as the union of $L$ disjoint polygonal subdomains $\oml$ with
	\begin{equation}\label{patch}
		{\Omega^\ell}= \cup_{K \in \Thk} K.
	\end{equation}
	\begin{assumption}\label{ass:macroel} We make the following assumptions:
		\begin{enumerate}[(i)] 
			\item the decomposition is geometrically conforming, that is, for each $\ell, m$,
			$\partial\Omega^\ell \cap \partial\Omega^m$ is either a vertex or a whole edge of both $\Omega^\ell$ and $\Omega^m$;
			
			\item there exists a positive constant $\gamma_2$ such that for each $\ell$ the distance between any two vertices of $\Omega^\ell$ is $\geq \gamma_2 \ \text{diam}(\Omega^\ell)$;

		\end{enumerate}
	\end{assumption}
	
	\begin{assumption}
		For all $\ell$, there exists a scalar $\rho_\ell > 0$ such that
		$\rho|_{\Omega^\ell} \simeq \rho_\ell$.
	\end{assumption}

	We let $\Gamma = \cup \partial\Omega^\ell \setminus \partial\Omega$ denote the skeleton of the decomposition. 
	For simplicity we assume that the decomposition is quasi uniform:
	
	\begin{assumption}
		There exists constants $\beta_0$ and $\beta_1$  such that for any two subdomains $\Omega^\ell$ and $\Omega^m$ we have 
		$\beta_0 \leq \text{diam}(\Omega^\ell)/\text{diam}(\Omega^m) \leq \beta_1$.
	\end{assumption}

	\begin{remark}
		We would like to point out that Assumption \ref{ass:macroel} is actually also an assumption on the tessellation $\Th$, satisfied, for instance, if $\Th$ is built by first introducing the $\Omega^ \ell$s and then refining them. The design and testing of FETI-DP and BDDC type preconditioners on more general tessellation will be the object of a future study. 
	\end{remark}

	We are interested here in explicitly studying the dependence of the estimates that we are going to
	prove on the number and size of the subdomains, the number and size of the elements of the tessellations, and on order $k$ of the virtual element method.
	To this end, in the following we will employ the notation
	$A \lesssim B$ (resp. $A \gtrsim B$) to say that the quantity $A$ is bounded from above (resp. from below)
	by $cB$, with a constant $c$ independent of $\rho$ and depending on the tessellation and on the decomposition only via the constants in Assumptions 2.1-3 and 3.1-3.
	The expression $A \simeq B$ will stand for $ A \lesssim B \lesssim A$.

	Observe that, in view of the quasi uniformity assumptions on the tessellation $\Th$ and on its decomposition into subdomains, we can introduce global mesh size parameters $H$ and $h$ such that for all $\ell$ and for all $K$
	\[
	h_K \simeq h, \qquad \text{diam}(\Omega^\ell) \simeq H.
	\]

	\

	In the following it will be useful to define several sets of ``pointers'' for identifying specific sets of nodes and edges.  To this end, recalling that \[\Upsilon = \{ y_i,\ i=1,\cdots,N \}\]
	denotes the set of nodes corresponding to boundary degrees of freedom, 
	we let 
	\[
	\mathcal{Y} = \{ i: \ y_i \in \Gamma\} 
	\]
	identify those nodes lying on the interface $\Gamma$ of the decomposition of $\Omega$ into subdomains. 
	For each subdomain $\oml$ we let $\Yl \subset \mathcal Y$
	\[
	\Yl = \{i \in \mathcal{Y}: y_i \in \partial\oml\}
	\]
	identify the set of nodes on the boundary of $\oml$.
	We let $\mathcal{X}$ and $\Xl$ identify, respectively, the set of cross-points and  the set of vertices of the subdomain $\oml$
	\[\mathcal{X} = \{i\in \mathcal{Y}: y_i \text{ is a cross point} \}, \qquad \Xl = \mathcal{X} \cap \Yl.\] 
	
	\
	
	Moreover, letting $\EH$ denote the set of those macro-edges $E$ of the decomposition which are interior to $\Omega$,  for each $E\in \EH$  we let
	\[ \YE= \{ i \in \mathcal{Y}: \ y_i \in \bar E  \}, \qquad \mathcal{X}_E = \{
	i \in \mathcal{X}: \ y_i \in \bar{E}
	\} \]
	identify respectively  the set of nodes and the set of cross points belonging to $E$.
	For each node $y_i$  we let $\Ni$ identify the set of the $n_i$ subdomains to whose boundary $y_i$ belongs 
	\[ \Ni = \{\ell: y_i \in \partial\oml\},\qquad n_i = \#(\Ni). \]
	Remark that for all $i \in \cY \setminus \cX$ we will have $n_i = 2$.
	
	\
	
	The subdomain spaces $V_h^\ell$ and bilinear forms $a_h^\ell:V_h^\ell \times V_h^\ell \to \mathbb{R}$ are defined, as usual, as
	\[ 
	V_h^\ell = {V_h}|_{\oml}, \qquad a_h^\ell(u_h,v_h) = \sum_{K\in \Thk} a_h^K(u_h,v_h),
	\label{ah}
	\]
	and we easily see that solving Problem \ref{discrete_full} is equivalent to finding $u_h = (u_h^\ell)_\ell \in \prod V_h^\ell$ minimizing
	\[
	J(u_h) = \frac 1 2 \sum_{\ell} a^\ell_h(u_h^\ell,u_h^\ell) - \int_\Omega f u_h
	\]
	subject to a continuity constraint across the interface. In view of (\ref{elemequiv}) we immediately obtain that for all $u, v \in V_h^\ell$
	\[
	a_h^\ell(u,v) \lesssim | u |_{1,\oml}  | v |_{1,\oml}, \qquad a_h^\ell(u,u) \simeq  | u |_{1,\oml}^2.
	\]
	
	\
	
	The FETI-DP and BDDC preconditioners are constructed according to the same strategy used in the finite element case. More precisely we let 
	\[
	\mathring V_h = \prod \mathring V_h^\ell,\quad \text{ with }\quad \mathring V_h^\ell = V_h^\ell \cap H^1_0(\oml),
	\]
	and
	\begin{equation}\label{Wspace}
		W_h = \prod_{\ell}W_h^\ell, \quad \text{ with }\quad W_h^\ell = {V_h^\ell}|_{\partial\oml}.
	\end{equation}
	Moreover, let $\widehat W_h \subset W_h$ denote the subset  of traces of functions which are continuous across $\Gamma$:
	\begin{equation}\label{What}
		\widehat W_h = \{ w_h \in W_h: \forall i \in \mathcal{Y}, \ell,m \in \mathcal{N}_i \Rightarrow w_h^\ell(y_i) = w_h^m(y_i)\}.
	\end{equation}
	Analogously we let $\widetilde W_h$ denote the subset of $W_h$ of traces of functions which are continuous at cross-points:
	\begin{equation}\label{Wtilde}
		\widetilde W_h = \{ w_h \in W_h: \forall i \in \mathcal{X}, \ell,m \in \mathcal{N}_i \Rightarrow w_h^\ell(y_i) = w_h^m(y_i)\} .
	\end{equation}

	On $W_h$ we define a norm and a seminorm:
	\[
	\normWh{w_h}^2 = \sum_{\ell}\| w_h^\ell \|_{H^{1/2}(\partial\oml)}^2, \qquad \snormWh{w_h}^2 = \sum_{\ell}| w_h^\ell |_{H^{1/2}(\partial\oml)}^2.
	\]
	Moreover, for each macroedge $E = \Gamma^\ell \cap \Gamma^m\in \EH$, we let $W^E_h = W^\ell_h|_E = W^m_h|_E$.  The space $W_h^E$ is the order $k$ one dimensional finite element space on the grid induced on $E$ by the tessellation $\mathcal{T}_h$, which, in view of Assumptions \ref{ass:uniform}, is quasi uniform of mesh size $h$. In particular, it satisfies standard inverse and direct inequalities (recall that we are using scaled norms and seminorms): for all $r,s$ with $0 < r \leq k+1$, $0 \leq s < \min\{3/2,r\}$, $w \in H^r(E)$ implies
	\begin{gather}
		\inf_{v_h \in W_h^E} \| w - v_h \|_{H^s(E)} \lesssim \left(\frac{h}{Hk}\right)^{r-s}  | w |_{H^r(E)},
		\label{direct}
	\end{gather}
	and for all $s,r$ with $0 \leq s \leq r < 3/2$, $w_h \in W_h^E$ implies
	\begin{gather}
		\| w_h \|_{H^r(E)} \lesssim \left(\frac h {H k^2}\right)^{s - r} \| w_h \|_{H^s(E)}. \label{inverse}
	\end{gather}
	
	\
	
	As usual, we define a local discrete lifting operator $\Lifthl: W_h^\ell \to V_h^\ell$ as
	\begin{eqnarray}
		a^\ell_h (\Lifthl w_h,v_h)=0, & \forall v_h \in \mathring V_h^\ell, \\
		\Lifthl w_h = w_h & \ \text{ on }\partial\oml.
	\end{eqnarray}
	The following proposition holds
	\begin{proposition}\label{prop:lifting}
		$\Lifthl$ is well defined, and it verifies 
		\[
		| \Lifthl w_h |_{1,\oml} \simeq | w_h |_{1/2,\partial\oml}.
		\]
	\end{proposition}
	
	\begin{proof}
		We start by remarking that 	there exists a linear operator $\Pscottzhang: H^1(\Omega^\ell) \to V^\ell_h$ such that, for $v \in H^{1+s}$, $0 \leq s \leq k$ one has 
		\begin{equation}\label{stimaSZ}
			\| v - \Pscottzhang v \|_{0,\Omega^\ell} + h | v - \Pscottzhang v |_{1,\Omega^\ell} \lesssim 
			h^{s+1} | v |_{s+1,\Omega^\ell} .
		\end{equation}
		Moreover $\Pscottzhang$ can be constructed in such a way that if $v|_{\partial\Omega^\ell} \in W_h^\ell$ one has $\Pscottzhang v = v$ on $\partial\Omega^\ell$.
		This is achieved by modifying the proof of Proposition 4.2 in \cite{Moraetal15} by replacing the Cl\'ement operator by a Scott-Zhang type operator, as constructed, for instance, in equation (4.8.6) in \cite{BrennerScott}. In view of this result it is not difficult to construct an operator $L_h^\ell: W_h^\ell \to V_h^\ell$ satisfying
		$L^\ell_h w_h |_{\partial\oml} = w_h$ and 
		\begin{equation}\label{stab1} | L^\ell_h w_h |_{H^1(\oml)} \leq | w_h |_{H^{1/2}(\partial\Omega^\ell)}.
		\end{equation}
		In fact, letting $w_h^{\mathcal{H}} \in H^1(\omk)$	denote the harmonic lifting of $w_h$, we can define  
		\[
		L_h^\ell w_h = \Pscottzhang (w_h ^{\mathcal{H}}). 
		\]	
		Equation~\eqref{stab1} follows from the stability of the continuous harmonic lifting and of $\Pscottzhang$.	
		We now observe that $ \Lifthl w_h - L_h^\ell w_h \in \mathring V_h^\ell$ satisfies
		\[
		a_h^\ell (\Lifthl w_h - L_h^\ell w_h,v_h) = -a_h^\ell(L_h^\ell w_h,v_h)\quad \forall v_h \in \mathring V^\ell_h.
		\]
		We can then write
		\begin{gather*}
			| \Lifthl w_h - L_h^\ell w_h |^2_{1,\oml} \lesssim a^\ell_h(\Lifthl w_h - L_h^\ell w_h,\Lifthl w_h - L_h^\ell w_h) = -a_h^\ell (L_h^\ell w_h,\Lifthl w_h - L_h^\ell w_h) \\ \lesssim
			| L_h^\ell w_h |_{1,\Omega^\ell}   | \Lifthl w_h  - L_h^\ell w_h |_{1,\Omega^\ell},
		\end{gather*}
		yielding
		\[
		| \Lifthl w_h - L_h^\ell w_h |_{1,\oml}  \lesssim | L_h^\ell w_h |_{1,\Omega} \lesssim | w_h |_{1/2,\Omega}.
		\]
		By triangular inequality we then obtain
		\[
		| \Lifthl w_h |_{1,\oml}  \lesssim | L_h^\ell w_h |_{1,\Omega^\ell}+| \Lifthl w_h  - L_h^\ell w_h|_{1,\Omega^\ell} \lesssim	| w_h |_{1/2,\partial\oml}.
		\]	
	\end{proof}
	
	\
	
	For $w_h = (w_h^\ell)_\ell \in W_h$ we let $\Lifth(w_h) = (\Lifthl w_h^\ell)_\ell$, 
	so that we can split $V_h$ as $\mathring V_h \oplus \Lifth \widehat W_h$. 
	
	\

	We can then define a bilinear for $s:W_h \times W_h \to \mathbb{R}$ as
	\[
	s(w_h,v_h) := \sum_\ell a^\ell_h (\Lifthl w_h,\Lifthl v_h).
	\]
	Problem \ref{discrete_full} can be split as the combination of two independent problems
	\begin{problem}\label{interior} Find $\mathring{u}_h \in \mathring V_h$ such that for all $v_h \in \mathring V_h$ 
		\[
		a_h (\mathring u_h,v_h) = f_h(v_h).
		\]
	\end{problem}
	
	\begin{problem}\label{schur} Find $w_h \in \hW$ such that for all $v_h \in \hW$
		\[
		s(w_h,v_h) = \langle f , \mathcal{L}_h v_h \rangle.
		\]
	\end{problem}
	
	The proof of the following proposition is trivial.
	\begin{proposition} For all $v_h, w_h \in W_h$ we have
		\begin{equation} \label{contcoercschur}
			s(v_h,w_h) \lesssim | w_h |_{1/2,*} | v_h |_{1/2,*}, \qquad s(w_h,w_h) \gtrsim | w_h |_{1/2,*}^2.
		\end{equation}	
	\end{proposition}

	\

	Following the approach of \cite{brennersung} we introduce the operators $\hSop:\hW \to \hW'$ and $\tSop: \tW \to \tW'$ defined respectively as
	\begin{equation}\label{S_tilde_hat}
		\langle \hSop w_h, v_h \rangle = s(w_h,v_h)\ \forall v_h \in \hW, \qquad  \langle \tSop w_h, v_h \rangle = s(w_h,v_h)\ \forall v_h \in \tW,
	\end{equation}
	and we let $\Rop: \hW \to \tW$ denote the natural injection operator.  We observe that  
	\[
	\hSop = \Rop^T \tSop \Rop.
	\]
	Problem \ref{schur} becomes 
	\begin{equation}\label{system_BDDC}
		\hSop u =f.
	\end{equation}

	\subsection*{The BDDC Preconditioner} In order to build the BDDC preconditioner, we introduce, as in \cite{Toselli.Widlund},  the weighted counting functions $\delta_\ell$ which are associated with each 
	$\partial \Omega^\ell$ and are defined for $\gamma\in[1/2,\infty)$ by a sum of contribution from $\Omega^\ell$ and its neighbors. For $i \in \mathcal{Y}^\ell$, we set
	\begin{equation}\label{deltai}
		\delta_\ell(y_i)= \frac{\sum_{j\in \Ni } \rho_j^\gamma}{\rho_\ell^\gamma}.
	\end{equation}
	As in \cite{Mandel_algebraic,Klawonn.Widlund.06}, 
	we define a scalar product (i.e.  a symmetric positive definite bilinear form) $d:W_h \times W_h \to \mathbb{R}$ defined as
	\begin{equation}\label{defscald}
		d(w_h,v_h) =  \sum_\ell \sum_{i \in \Yl} d^{\ell,i} v^\ell(y_i)w^\ell(y_i),
	\end{equation}
	where the scaling coefficients $d^{\ell,i}$ defined as 
	\begin{equation}\label{dli}
		d^{\ell,i} = (\delta_\ell(y_i))^{-1} = \frac{\rho_\ell^\gamma}{\sum_{j\in \Ni } \rho_j^\gamma}.  
	\end{equation}

	We next introduce the projection operator $\ED: \tW \to \hW$, orthogonal  with respect to the scalar product $d$:
	\begin{equation}\label{ED}
		d(\ED w_h,v_h) = d(w_h, v_h), \ \forall v_h \in \hW.
	\end{equation}
	
	\smallskip
	
	With this notation we can define the BDDC method, as the  preconditioned conjugate gradient  method applied to system \nref{system_BDDC};
	the BDDC preconditioner $\BDDC: \hW'\to \hW$ takes the form 
	\begin{equation}\label{precBDDC}
		\BDDC = \ED \tSop^{-1} \ED^T. 
	\end{equation}
	
	\
	
	\subsection*{The FETI-DP Preconditioner} Following \cite{brennersung}, we introduce the quotient space $\Sigma_h = \tW / \hW$. We let $\Lambda_h = \Sigma_h'$ denote its dual and we let
	$\Bop: \tW \to \Sigma_h = \Lambda_h'$ be the quotient mapping, defined as 
	\[
	\Bop w_h = w_h +\hW.
	\]
	Observe that two element $w_h$ and $v_h$ are representative of the same equivalence class in $\Sigma_h$ if and only if they have the same jump across the interface: for $E \in \EH$ with $E = \Gamma^m \cap \Gamma^\ell$ $w_h^m - w_h^\ell = v_h^m - v_h^\ell$. We can then identify $\Sigma_h$ with the set of jumps of elements of $W_h$. The quotient map $\Bop$ is then the operator that maps an element of $W_h$ to its jump on the interface.

	\

	Clearly
	\[ \hW = \ker(\mathcal{B}) = \{  w_h \in \tW: \ b(w_h,\lambda)=0,\ \forall \lambda\in \Lambda_h \},\]
	where $b: \tW \times \Lambda_h \to \mathbb{R}$ is defined  as $b(w_h,\lambda_h) = \langle \Bop(w_h),\lambda_h \rangle$.
	Problem \ref{discrete_full} is then equivalent to the following saddle point problem: find $w_h \in \tW$, $\lambda_h \in \Lambda_h$ solution to
	\[
	\tSop w_h - \Bop^T \lambda_h = \tilde f,\qquad \Bop w_h = 0,
	\]
	where $\tilde f \in W_h'$ is defined by 
	\[
	\langle \tilde f , w_h \rangle = \int_{\Omega} f \Lifth w_h = \sum_{\ell} \int_{\oml} f \Lifthl w_h^\ell.
	\]
	Eliminating $w_h$ we obtain a problem in the unknown $\lambda_h \in \Lambda_h$ of the form
	\begin{equation}\label{system_FETI}
		\Bop \tSop^{-1}\Bop^T \lambda_h = -\Bop\tSop^{-1}\tilde f.
	\end{equation}

	The SPD operator $\Bop \tSop^{-1}\Bop^T$ is the one that the FETI-DP preconditioner tackles.  Letting $\Bop^+ : \Sigma_h \to \tW$ be any right inverse of $\Bop$ (for $\eta \in \Sigma_h$, $\Bop^+ \eta$ is any element in $\tW$ such that $\Bop \Bop^+ \eta = w_h$), we set
	\[
	\Bop^T_D = (\IdtW-\ED) \Bop^+,
	\]
	where $\IdtW$ denotes the identity operator in the space $\tW$.
	Remark that the definition of $\Bop_D^T$ is independent of the actual choice of $\Bop^+$. Indeed for $v_1$ and $v_2$ such that $\Bop v_1 = \Bop v_2$ we have $\Bop_D^T v_1 - \Bop_D^T v_2 = (v_1-v_2) - \ED (v_1-v_2) = 0$ where the last equality derives from the fact that since $\Bop v_1=\Bop v_2$, we have $v_1-v_2 \in \hW$.
	
	\
	
	The preconditioner $\FETI: \Sigma_h \to \Lambda_h$  for \nref{system_FETI} takes the form
	\begin{equation}\label{prec_FETI}
		\FETI = \Bop_D \tSop \Bop_D^T.
	\end{equation}

	\section{Analysis of the two preconditioners}
	
	Let us now provide bounds on the condition number of the preconditioned systems.  We observe that the space $W_h$ coincides with the analogous space that we would get by applying the same procedure in the finite element case with a triangular or polyhedral mesh sharing with $\Th$ the nodes on the interface. It is then reasonable to expect that by defining BDDC and FETI-DP preconditioners as in the finite element case we would obtain the same estimates on the dependence of the condition number on $H$, $h$ and $k$. 
	A possible way of analyzing the method could be to introduce an auxiliary standard finite element mesh inducing the same trace spaces on $\Gamma$, and to prove  that this can be done while satisfying the assumptions needed for the analysis of the preconditioners in the FEM case. This would allow to carry over the results for the FEM method to the VEM method. Here, however, we prefer to start from scratch and, in view of a future extension to the three dimensional case, providing a proof that will be mostly independent on the properties of the tessellation inside the subdomains.  
	We start by remarking that, by construction, we are in the framework of \cite{MandelSousedik}, yielding the equivalence between the BDDC and the FETI-DP preconditioners. In particular we have the identity
	\[
	\Bop_D^T \Bop + \ED = \IdtW.
	\]
	In fact, it is not difficult to see that for all $w_h \in \tW$ we have that $(\IdtW  - \Bop_D^T \Bop)w_h \in \hW$ and that we have \[d((\IdtW  - \Bop_D^T \Bop)w_h,v_h) = d(w_h , v_h)  \quad \text{ for all }v_h \in \hW.\]
	
	Then, letting 
	\[
	\omega_{FETI-DP} = \kappa(\FETI(\Bop \tSop^{-1} \Bop^T)), \qquad  \omega_{BDDC}= \kappa(\BDDC \hSop),
	\]
	denote the condition numbers of the operators $\FETI(\Bop \tSop^{-1} \Bop^T)$ and $\BDDC \hSop$, we have (\cite{MandelSousedik})
	\begin{equation}\label{minimal1}
		\omega_{BDDC} \lesssim \max_{w_h \in \tW} \frac {s(\ED w_h, \ED w_h)} {s(w_h,w_h)}, 
		\qquad \omega_{FETI-DP} \lesssim \max_{w_h \in \tW} \frac {s(\Bop_D^T\Bop w_h, \Bop_D^T\Bop w_h)} {s(w_h,w_h)}, 
	\end{equation}
	and
	\begin{equation}\label{minimal2}
		\omega_{BDDC} \simeq \omega_{FETI-DP}.
	\end{equation}
	In order to have a bound on both condition numbers, we then only need to bound $s(\ED w_h, \ED w_h)$ in terms of $s(w_h,w_h)$. 
	
	\
	
	We start by proving a technical Lemma. For $w_h \in W_h$ let us consider the splitting
	\begin{equation}\label{splitting}
		w_h = w_H + \mathring w_h,
	\end{equation}
	where $w_H^\ell(y_i) = w^\ell_h(y_i)$ for all $i \in \Xl$, and $w_H^\ell(y_i) = 0$ for all $i \in \Yl\setminus \Xl$ ($w_H^\ell$ coincides with $w_h^\ell$ at the vertices of  $\oml$ and it is zero at all other nodes on $\partial\oml$). 
	We have the following Lemma (\cite{Smith91,BPSIV}), for which we present a new proof that only relies on the properties of space $W_h$ and is independent of the properties of the tessellation in the interior of the subdomains $\Omega^\ell$.
	\begin{lemma}\label{lem:technical}
		For all $w_h \in W_h^E$, $E \in \EH$, letting 
		$\mathring w_h \in W_h^E$ be defined by $\mathring w_h(y_i) = w_h(y_i)$ for all $i \in \YE \setminus \XE$, and $\mathring w_h(y_i) = 0$ for $i \in \XE$,
		we have
		\[
		\| \mathring w_h \|_{H^{1/2}_{00}(E)} \lesssim (1 + \log(Hk^2/h)) | w_h |_{H^{1/2}(E)} + \sqrt{1 + \log(Hk^2/h)} \| w_h \|_{L^\infty(E)}.
		\]
	\end{lemma}

	\begin{proof} 
		We start by proving the result for the lowest order VEM method ($k=1$). We let  $\mathring W_h^E$ be defined as
		\[
		\mathring W_h^E= \{ w_h \in W_h^E: \ w_h(y_i)=0 \text{ for all }i\in \XE \}.
		\]
		Let $\pi_h: L^2(E) \to W_h^E$ and $\pi_h^0: L^2(E) \to \mathring W_h^E$ denote  the $L^2$-projection onto respectively $W_h^E$ and $\mathring W_h^E$.
		Moreover, let $i_h^0: W_h^E \to \mathring W_h^E$ be defined by
		\[
		i_h^0 w_h (y_i) = w_h(y_i) \ \text{ for all } i \in \YE\setminus\XE.
		\]
		
		Remark that for $w_h \in W_h^E$ we have
		\[
		\| i_h^0 w_h \|_{L^2(E)}^2 \lesssim h \sum_{i\in \YE\setminus\XE} |w_h(y_i)|^2 \lesssim h \sum_{i\in \YE} |w_h(y_i)|^2 \lesssim  \| w_h \|^2_{L^2(E)}.
		\]
		Consider now the operator $\pi_h^1 = i_h^0 \circ \pi_h : L^2(E) \to \mathring W_h^E$. We easily see that $\pi_h^1$ is $L^2$ bounded: for all $w \in L^2(E)$,
		\[
		\| \pi_h^1 w \|_{L^2(E)} \lesssim \| \pi_h w \|_{L^2(E)} \lesssim \| w \|_{L^2(E)}.
		\]
		On the other hand, observing that $i_h^0 \circ \pi_h^0 = \pi_h^0$, we have that, for $w \in H^1_0(E)$
		\begin{align*}
			\| w - \pi_h^1 w \|_{L^2(E)} &= \| w - \pi_h^0 w +
			i_h^0 \pi_h^0 w - i_h^0 \pi_h w \|_{L^2(E)} \lesssim \\
			&\lesssim
			\| w - \pi_h^0 w \|_{L^2(E)} + \| \pi_h^0 w - \pi_h w \|_{L^2(E)}.
		\end{align*}
		By adding and subtracting $w$ in the second term on the right hand side we obtain, for $w \in H^1_0(E)$,
		\begin{gather*}
			\| w - \pi_h^1 w \|_{L^2(E)} \lesssim
			\| w - \pi_h^0 w \|_{L^2(E)} + \| w - \pi_h w \|_{L^2(E)} \lesssim
			\frac h H | w |_{H^1(E)}.
		\end{gather*}
		This allows us to prove, by a standard argument, that $\pi_h^1$ is $H^1_0$-bounded. In fact, letting $\Pi_h^1: H^1_0(E) \to \mathring W_h^E$ denote the $H^1_0(E)$ projection onto $\mathring{W}_h^E$, defined as
		\[
		\Pi_h^1(w)	= \arg \min_{v_h \in \mathring{W}_h^E} \left(\frac 1 2 | v_h |_{1,E}^2 - \langle w, v_h \rangle\right),
		\]
		we have
		\begin{align*}
			| \pi_h^1 w |&_{H^1(E)} \lesssim \\
			&\lesssim
			| w |_{H^1(E)} +\left(\frac h H\right)^{-1}\| \pi_h^1 w -
			\Pi_h^1 w \|_{L^2(E)}\lesssim \\
			&\lesssim
			| w |_{H^1(E)} + \left(\frac h H\right)^{-1} \| \pi_h^1 w -
			w \|_{L^2(E)} + \left(\frac h H\right)^{-1} \| w - \Pi_h^1
			w \|_{L^2(E)} \lesssim\\
			& \lesssim | w |_{H^1(E)} +\left(\frac h H\right)^{-1}
			\left(\frac h H\right) | w |_{H^1(E)} + \left(\frac h H\right)^{-1}
			\left(\frac h H\right) | w |_{H^1(E)}\lesssim 
			| w |_{H^1(E)}.
		\end{align*}
		By space interpolation we then deduce that $\pi_h^1$ is uniformly $H^s_0(E)$ bounded for all $s \in [0,1]$, $s\not = 1/2$, that is for all $w \in H^s_0(E)$ we have
		\[
		\| \pi_h^1 w \|_{H^s_0(E)} \lesssim \|  w \|_{H^s_0(E)},
		\]
		with a constant independent of $s$, whereas for $w \in H^{1/2}_{00}(E)$ we have
		\[ \| \pi_h^1 w \|_{H^{1/2}_{00}(E)} \lesssim  \| w \|_{H^{1/2}_{00}(E)}.\]
		
		\
		{We now recall that for $0< s< 1/2$, the space $H^s(E)$ is embedded in $H^s_0(E)$ and we have (\cite{Bsub3field})
			\[
			\| w \|_{H^s_0(E)} \lesssim \frac {1}{1-2s} | w |_{H^s(E)} + \frac {1} {\sqrt{1-2s}} \| w \|_{L^\infty(E)}.
			\]}
		Then,  for $\varepsilon \in
		]0,1/2[$ and  $w_h \in W_h^E$ we have
		\begin{align*}
			\| \pi_h^1 w_h \|_{H^{1/2}_{00}(E)} &\lesssim \left(\frac h H\right)^{-\varepsilon} \| \pi_h^1 w_h
			\|_{H^{1/2-\varepsilon}_0(E)} \lesssim
			\left(\frac h H\right)^{-\varepsilon}
			\| w_h \|_{H^{1/2-\varepsilon}_0(E)} \lesssim \\
			&  \lesssim \frac {1}{\varepsilon} \left(\frac{h}{H}\right)^{-\varepsilon} |
			w_h |_{H^{1/2-\varepsilon}(E)} 
			+ \frac {1}{\sqrt{\varepsilon}} \left(\frac{h}{H}\right)^{-\varepsilon} \|
			w_h \|_{L^\infty(E)} \\
			&	\lesssim  (1 + \log(H/h)) | w_h
			|_{H^{1/2}(E)} + \sqrt{1+\log(H/h)} \| w_h \|_{L^{\infty}(E)},
		\end{align*}
		where the last inequality is obtained by choosing
		$\varepsilon = 1/ |1+\log(H/h)|$ .
		Observing that for $w_h \in  W_h^E$ we have
		\[
		\mathring w_h^\ell|_E =  i_h^0 w_h|_E = \pi_h^1 w_h|_E,
		\]
		we immediately get the thesis for $k=1$. For $k=2$ the result is proven analogously. In order to prove the result for $k \geq 3$, we proceed as in \cite{Toselli.Widlund}, Section 7.4.1. We let $w_h^\star$ and $\mathring w_h^\star$ denote the 
		two functions which are piecewise linear on the grid with mesh size $\gtrsim h^\star = h/k^2$ induced on $E$ by the nodes	$y_i$, $i \in \YE$, and interpolating respectively
		$w_h$ and $\mathring w_h$ at such nodes. We know (see \cite{Canuto.1994}) that the following equivalences hold
		\[
		\| w_h^\star \|_{0,E} \simeq \| w_h \|_{0,E}, \text{ and }\ 	\| w_h^\star \|_{1,E} \simeq \| w_h \|_{1,E},
		\]
		whence
		\[
		| w_h^\star |_{H^{1/2}(E)} \simeq 	| w_h |_{H^{1/2}(E)} \text{ and }\ 	| \mathring w_h^\star |_{H^{1/2}_{00}(E)} \simeq | \mathring w_h^\star |_{H^{1/2}_{00}(E)}.
		\]
		Then, applying the Lemma for $k = 1$ and $\mathring w_h^\star$ and $w_h^\star$, we obtain
		\begin{align*}
			\| \mathring w_h \|_{H^{1/2}_{00}(E)} &\lesssim 	\| \mathring w^\star_h \|_{H^{1/2}_{00}(E)} \lesssim \\
			&	\lesssim (1 + \log(H/h^\star)) | w_h^\star |_{H^{1/2}(E)}  + 
			\sqrt{ 1 + \log(H/h^\star)} \| w_h^\star \|_{L^{\infty(E)}}\\
			&	\lesssim (1 + \log(Hk^2/h)) | w_h |_{H^{1/2}(E)}  + \sqrt{ 1 + \log(Hk^2/h)} \| w_h \|_{L^{\infty(E)}}.
		\end{align*}
	\end{proof}

	We are now able to prove the following lemma.
	\begin{lemma} \label{lem:5.3} For all $w_h \in \tW$ we have
		\begin{equation}\label{faclogE}  s( \ED w_h,\ED w_h) \lesssim (1 + \log(Hk^2/h))^2 s(w_h,w_h). \end{equation}
	\end{lemma}
	
	\begin{proof}  We start by observing that $v_h = \ED w_h$ is uniquely defined by its single values $v_i$ at the nodes $y_i$, $i\in \mathcal{Y}$. A straightforward computation yields, for $i \in \mathcal{Y}$:
		\begin{equation}\label{defthetai}
			v_i = \theta^{-1}_i \sum_{k \in \Ni} \rho^\gamma_k w^m_h(y_i), \qquad \theta_i=\sum_{k\in \Ni} \rho_k^\gamma.
		\end{equation}

		Using the splitting (\ref{splitting}) for both $w_h$ and $v_h$ and observing that $v_H = w_H$, we have that
		\[
		\sum_{\ell} | w_h - v_h |^2_{1/2,\Gamma_\ell} \lesssim  \sum_{\ell} | \mathring w_h^\ell - \mathring v_h |^2_{1/2,\Gamma_\ell}.
		\]

		\
		
		We observe that, for a given macroedge $E = \partial\oml \cap \partial\omk$, we have
		\[ \mathring {v}_h{|_E} = (\rho_\ell^\gamma+\rho_m^\gamma)^{-1}(\rho_\ell^\gamma { \mathring w_h^\ell}|_{E}+ \rho_m^\gamma {\mathring w_h^m}|_{E}),\]
		whence
		\[
		(\mathring{w}_h^\ell -  \mathring {v}_h){|_E} = (\rho_\ell^\gamma+\rho_m^\gamma)^{-1} \rho_m^\gamma (\mathring w_h^\ell -\mathring w_h^m).
		\]
		
		Then, recalling that	\begin{equation}\label{boundrhogamma}
			\rho_\ell ((\rho_\ell^\gamma + \rho_m^\gamma)^{-1}\rho_m^\gamma)^2 \leq \min\{\rho_\ell,\rho_m\} \leq \rho_m,
		\end{equation} we can write
		\begin{align*}
			\sum_{\ell}\rho_\ell | \mathring w_h^\ell - \mathring v_h & |^2_{1/2,\Gamma^\ell} \lesssim \sum_{\ell} \sum_{E \in \El} \rho_\ell   | \mathring w_h^\ell - \mathring v_h |_{H^{1/2}_{00}(E)}^2 \\&\lesssim
			\sum_{\ell} \sum_{E = \partial\oml \cap \partial\omk \in \El} \rho_\ell(\rho_\ell^\gamma+\rho_m^\gamma)^{-2} \rho_m^{2\gamma} | \mathring w_h^\ell - \mathring w_h^m |_{H^{1/2}_{00}(E)}^2  \\ &=
			\sum_{E = \partial\oml \cap \partial\omk \in \mathcal{E}} 
			(\rho_\ell^\gamma+\rho_m^\gamma)^{-2}(\rho_\ell \rho_m^{2\gamma} + \rho_m \rho_\ell^{2\gamma} ) | \mathring w_h^\ell - \mathring w_h^m |_{H^{1/2}_{00}(E)}^2 \\
			&\lesssim 	\sum_{E = \partial\oml \cap \partial\omk \in \mathcal{E}} \min\{
			\rho_\ell,\rho_m\} | \mathring w_h^\ell - \mathring w_h^m |_{H^{1/2}_{00}(E)}^2 .
		\end{align*}
		We now recall that we have the following bound for any $w_h\in W_h^E$ with $w_h(x) = 0$  for some $x \in E$
		\begin{equation}\label{Bsubstr34}
			\| w_h \|_{L^{\infty}(E)} \lesssim \sqrt{(1+\log(Hk^2/h))} | w_h |_{H^{1/2}(E)}.
		\end{equation}
		This bound was proven for $k = 1$ in \cite{BPSI} as well as in \cite{Bsub3field}, where a proof only relying on the properties of $W_h^E$ was provided. It is straightforward to adapt such a proof to cover the case $k>1$.	Using Lemma \ref{lem:technical}, as well as (\ref{Bsubstr34})
		(which we can apply since $w_h^\ell - w_h^m$ vanishes at the extrema of $E$) we have
		\begin{align*}
			| \mathring w_h^\ell - \mathring w_h^m |_{H^{1/2}_{00}(E)}^2  &\lesssim (1 + \log(Hk^2/h))^2 | w_h^\ell -w_h^m|^2_{H^{1/2}(E)} + \\
			&  + (1 + \log(Hk^2/h)) \| w_h^\ell - w_h^m \|^2_{L^\infty(E)}	\lesssim  \\ &
			\lesssim (1 + \log(Hk^2/h))^2 | w_h^\ell -w_h^m|^2_{H^{1/2}(E)}\lesssim   \\& \lesssim (1 + \log(Hk^2/h))^2(| w_h^\ell|_{H^{1/2}(E)}^2+| w_h^m|_{H^{1/2}(E)}^2),
		\end{align*}
		whence we	easily obtain
		\begin{align*}
			\sum_{\ell}\rho_\ell | \mathring w_h^\ell - \mathring v_h |^2_{1/2,\Gamma^\ell} & \lesssim (1 + \log(Hk^2/h))^2 \sum_\ell \rho_\ell | w_h^\ell |_{H^{1/2}(\Gamma_\ell)}\lesssim \\
			&  \lesssim (1 + \log(Hk^2/h))^2 s(w_h,w_h).
		\end{align*}
	\end{proof}

	In view of (\ref{minimal1}) and (\ref{minimal2}) we have the following corollary
	
	\begin{corollary}\label{cond_bound}
		The following bound on the condition numbers for BDDC and FETI-DP holds
		\[
		\omega_{BDDC} \lesssim (1+\log(Hk^2/h))^2,\qquad \omega_{FETI-DP} \lesssim (1+\log(Hk^2/h))^2.
		\]
		
	\end{corollary}

	\

	\newcommand{\auxV}{V_{aux}}
	
	\subsubsection{Extension to other $H^1$ conforming VEM classes}\label{sec:extension} While for simplicity we considered in our analysis the simpler early version of the VEM method, where the local  space is defined by (\ref{deflocalspace}), 
	it is not difficult to realize that the analysis presented here carries over to other $H^1$ conforming versions of the method. In particular, several VEM methods make use of an auxiliary enlarged local space $\auxV^{K,k}$ defined as
	\[
	\auxV^{K,k} =   \{ v \in H^1(K):\ v|_{\partial K} \in \mathbb{B}_k(\partial K),\ \Delta v \in \mathbb{P}_{k}(K)\},
	\]
	and define the local space $V^{K,k}$ as the subspace obtained by requiring that a certain number of degrees of freedom, corresponding to high order moments, are a function of the remaining ones. This is the case of the VEM space used for treating problems with coefficients that vary within each element \cite{beirao_var_coef}, and of the nodal serendipity VEM \cite{beirao_serendipity}. 
	As also for these new spaces it holds that $\auxV^{K,k}|_{\partial K} = \mathbb{B}_k(\partial K)$, it is not difficult to realize that  our analysis carries  over to such cases, provided we can verify that  property (\ref{elemequiv}) and  Proposition \ref{prop:lifting} hold. These are indeed the only steps in our analysis which depend on the definition of $V^{K,k}$ within the elements. 
	Observe that  property (\ref{elemequiv}) is one of the basic requirements in the construction of the VEM stabilized operator (see, e.g. (5.4) and (5.23) in \cite{beirao_var_coef}). As far as Proposition \ref{prop:lifting}, its proof relies on the existence of a Scott-Zhang type operator satisfying an estimate of the form (\ref{stimaSZ}). The construction of such an operator, as proposed in \cite{Moraetal15} for the case here considered, carries over to other $H^1$ conforming VEM spaces, yielding (\ref{stimaSZ}) with constants possibly depending on $k$. It is in fact not difficult to verify that adapting such a construction to different definitions of the local space $V^{K,k}$ yields an operator that locally preserves polynomials of order $k$ (provided of course that they are contained in the local space).

	\begin{remark}
		As far as the possible extension to other PDEs or to non $H^1$ conforming methods
		such as $H(div)$ conforming or discontinuous Galerkin, the same issues arise in the virtual element framework as in the finite element case, and, in view of the analysis performed here, we believe that the same solution strategies apply, tough the treatment of such cases is not immediate and needs further work. 
	\end{remark}

	\section{Realizing the preconditioner}\label{algebraic}
	
	As already observed, both the  FETI-DP and the BDDC preconditioner for the virtual element method are the same as for the finite element method, and a vast literature exists on how such methods are efficiently implemented, that carries over, essentially as it is, to our case. Nevertheless, for the sake of those reader with little experience in domain decomposition,  we recall some details of the implementation, which we adapt to the notation used in the previous sections. 
	
	\subsection{Functions} Functions in the different spaces  are represented algebraically as follows.
	\begin{itemize} 
		\item functions $u^ \ell_h \in W_h^\ell$ are represented by the vectors $\vec{u}^ \ell \in \mathbb{R}^\Nl$ (with $\Nl = \#(\Yl)$) of their values at the nodes $y_i, i \in \Yl$;
		\item functions $u_h = (u_h^\ell)_\ell \in W_h$ will be represented by the vectors obtained by concatenating the $\vec u^\ell$s: $\vec u = (\vec u^1,\cdots,\vec u^\ell) \in \mathbb{R}^N$, $N = \sum_{\ell} \Nl$;
		\item functions $\hat u_h \in \hW$ will be represented by the vectors $\hat {\vec u} \in \mathbb{R}^{\hN}$ ($\hN = \#(\mathcal{Y})$) of their values at the nodes on $\Gamma$;
		\item functions $\tilde u_h \in \tW$ will be represented by the vectors $\tilde{\vec u} \in \mathbb{R}^{\tN}$,  of their single values at the crosspoints plus their double value at each of the nodes $y_i \in \mathcal{Y}\setminus\mathcal{X}$;
		\item the elements of $\Sigma_h$ ({\em i.e.} the jumps of functions in $\tW$) will be represented by vectors $\vec s \in \mathbb{R}^M$, $M = \#(\mathcal{Y}\setminus\mathcal{X})$, of jump values at the nodes $y_i$, $i \in \mathcal{Y} \setminus \mathcal{X}$. 
	\end{itemize}
	
	\noindent Corresponding to the splitting \nref{splitting}, the vectors $\widetilde{\vec u}  \in\Re^{\tN}$ are split as 
	\begin{equation}\label{discsplitting}
		\widetilde{\vec u} = [\vec u_{\Delta},\vec{u}_{\Pi}],
	\end{equation} 
	with $\vec u_\Pi \in \mathbb{R}^{\NP}$ being the vector of values at the cross points (primal nodes), $u_\Delta^\ell$ being the vector of values of $u_h^\ell$ at the points $y_i$, $i\in \Yl\setminus \Xl$ (dual nodes) and $\vec u_\Delta = (u_\Delta^1,\cdots,u_\Delta^L) \in \mathbb{R}^{\ND}$ . 
	Observe that, with the choices made in Section \ref{sec:dualprimal}, we have  $\ND = 2 \#(\mathcal{Y}\setminus\mathcal{X})$, $\NP = \#(\mathcal{X})$, and $N = \ND + \sum_{i \in \mathcal{X}} n_i$. 
	
	\subsection{Operators} In order to implement the BDDC and the FETI-DP methods, we need to construct (or implement the action of) matrices corresponding to the different operators introduced in Section \ref{sec:dualprimal}, and, in particular, to 
	$\ED$ and $\B^T_D$ and $\tSop^{-1}$. 
	
	\
	
	We start by letting $S_\ell \in \Re^{\Nl \times \Nl}$ denote the  $\Nl \times \Nl$ matrix realizing the operator
	$\mathcal{S}^\ell$, and we let $S \in \Re^{N\times N}$, defined by  
	\[
	S = \left(
	\begin{array}{ccc}
		S_1 & & \\
		& \ddots & \\
		& & S_L
	\end{array}
	\right),
	\]
	be the block diagonal matrix realizing the operator $\mathcal{S}$. Observe that the matrix $S_\ell$ is the Schur complement, taken with respect to the degrees of freedom on $\partial\Omega^\ell$, of the stiffness matrix $A_\ell$ discretizing the bilinear form $a_h^\ell$ on the space $V_h^\ell$ defined in \nref{ah}.
	Applying $S^\ell$ implies then solving by the VEM method a local version of problem (\ref{variational}) with non homogeneous boundary condition.
	
	\
	
	We let $\hR \in \Re^{N\times \hN} $ and $\tR \in \Re^{N \times \tN}$ denote the two matrices realizing the natural injection of $\hW$ and $\tW$ into $W_h$, respectively, and let $R$ denote the matrix realizing the injection of $\hW$ into $\tW$. These are matrices with $0$ and $1$ entries, whose action essentially consists in copying the value of the degrees of freedom in the right positions. With this notation, the matrices $\tS \in \Re^{\tN\times\tN}$ and $\hS \in \Re^{\hN\times \hN}$, realizing the operators $\tSop$ and $\hSop$, respectively, take the form
	\[
	\tS = \tR^T S \tR, \qquad \hS = \hR^T S \hR,
	\]
	and we have 
	\[
	\hS = R^T \tS R.
	\]

	\
	
	We let $D_\ell \in \Re^{\Nl \times \Nl}$ be the diagonal matrix whose diagonal entry on the line corresponding to the node $y_i, i \in \Yl$  is equal to the scaling coefficient 
	$d^{\ell,i} = 
	\rho_\ell^\gamma / 
	\sum_{j\in \Ni } \rho_j^\gamma$, and we assemble  $D \in \Re^{N\times N}$ as
	\[
	D = \left(
	\begin{array}{ccc}
		D_1 & & \\
		& \ddots & \\
		& & D_L
	\end{array}
	\right),
	\] 
	so that the scalar product $d: W_h \times W_h \in \Re$ defined by (\ref{defscald}) is realized by the matrix $D$. We easily check that we have $\hR^T D \hR = I_{\hN}$ (where $I_{\hN}$ denotes the $\hN \times \hN$ identity matrix). It is not difficult to verify that the matrix $E_D \in \Re^{\hN\times \tN}$ realizing the projector $\ED$ is
	\[
	E_D = \hR^T D \tR \in \Re^{\hN \times \tN}.
	\]
	In fact, letting $\widehat{\vec w}$ denote the vector of coefficients of  $\ED \widetilde u_h$, we have
	\[
	\widehat{\vec v}^T  \widehat {\vec w} = \widehat{\vec v}^T \, \hR^T D \hR \, \widehat {\vec w} = \widehat{\vec v}^T \, \hR^T D \tR \,\tilde{\vec u}, \qquad \text{ for all } \widehat{\vec v} \in \Re^{\hN}.
	\]
	
	\
	
	Let us now come to the construction of the matrices corresponding to the jump operator $\B$ and its pseudo inverse $\B_D^T$.  We let $B = [B_1, \dots, B_L] \in \Re^{M \times \tN}$ denote the matrix with entries in the set $\{-1,0,1\}$, representing the jump operator.  Observe that we have $BB^T = 2 I_M$, that is, in our case, $B^T/2$ is a right inverse of $B$, and this gives us the operator $\B^+$. Then $B^T_D \in \Re^{\tN \times M}$ can be defined as
	\[
	B^T_D = (I_{\tN} - R E_D)B^T/2.
	\] 
	By direct computation it is not difficult to verify that $\BD$ is a scaled version of B
	$$\BD = [ \DB_1 B_1, \dots, \DB_L B_L] $$
	where $\DB_\ell, l = 1,\dots, L$, are diagonal scaling matrices with entries related to the neighboring subdomain of $\Omega_\ell$.
	More precisely, each row of $B_\ell$ with a non zero entry corresponds to a node $y_i$, $i \in \Yl \setminus \Xl$, which verifies $y_i \in \partial \Omega_\ell \cap \partial \Omega_k$ for a unique $k$. The diagonal entry of $D^\flat_\ell$ corresponding to such node takes the value $d^{k,i} = \rho_k^\gamma/ \sum_{j\in \Ni } \rho_j^\gamma$.
	
	Then the system we deal with the BDDC method is
	\begin{equation}\label{system_BDDC_alg}
		\hS \hat{\vec u}  = \hat {\vec f}, 
	\end{equation}
	preconditioned with
	\begin{equation}\label{prec_BDDC_alg}
		M_{BDDC} = E_D \tSmat^{-1} E_D^T,
	\end{equation}
	matrix counterpart of \nref{system_BDDC}  and \nref{precBDDC} respectively.
	
	\
	
	The last matrix whose action we need to evaluate is $\tSmat^{-1}$, corresponding to the inverse of the operator $\tSop$. Different approaches have been proposed for its efficient implementation \cite{Klawonn2016,Klawonn2015,Klawonn.Rheinbach.10}. We start by observing that, corresponding to the splitting \nref{discsplitting}, the matrix $\tSmat$ can be rewritten as
	\[
	\widetilde S = \left( 
	\begin{array}{cc}
		\widetilde S_{\DD} & \widetilde S_{\DP}\\
		\widetilde S_{\Pi\Delta} & \widetilde S_{\PP}
	\end{array}
	\right) .
	\]
	The matrix $\widetilde S_\DD$ is  block diagonal  and invertible. Following \cite{LiWidlund2006}, we can use block Cholesky elimination to get
	\begin{equation}\label{invertStilde}
		\tSmat^{-1} = 
		\left(
		\begin{array}{cc}
			\widetilde S_{\DD}^{-1} & 0 \\
			0 & 0
		\end{array}\right) + \left(
		\begin{array}{c}
			-\widetilde S_{\DD}^{-1} \widetilde S_{\DP}\\
			I_{N_\Pi}
		\end{array}
		\right)
		\left( \widetilde  S_{\PP} - \widetilde {S}_{\Pi\Delta} \widetilde {S}^{-1}_{\DD} \widetilde {S}_{\DP}\right)^{-1} \left(
		\begin{array}{cc}
			-\widetilde  S_{\Pi\Delta} \widetilde S^{-1}_{\DD} & I_{N_\Pi}\end{array}
		\right)
	\end{equation}
	Only the matrix $\tilde S_{\PP} - \tilde{S}_{\Pi\Delta} \tilde{S}^{-1}_{\DD} \tilde{S}_{\DP}$ is assembled once for all (this is done by computing its action on the columns of $I_{N_\Pi}$). As far as the other matrices on the right hand side of the expression are concerned, only their action on a given vector needs to be implemented.

	Each block $S_{\DD}^\ell$ is the Schur complement with respect to the degrees of freedom interior to the subdomain $\Omega_\ell$ of the matrix 
	$A_0^\ell$  obtained from the local stiffness matrix  $A^\ell$ by eliminating the rows and columns corresponding to degrees of freedom attached to the cross points. In order to apply $\tilde S_{\DD}^{-1}$ to a vector $\vec r$ one can then solve a system with the diagonal matrix\ \[A_0 = \left( 
	\begin{array}{ccc}
		A^1_0 & & \\
		& \ddots & \\
		& &A^L_0
	\end{array}
	\right). \]
	More precisely, 
	\[
	\tSmat_{\DD}^{-1} = C^T
	{A_0}
	^{-1} C,
	\] 
	where $C \in \Re^{\Ndofs \times \tN}$ (where $\Ndofs$ is the total number of degrees of freedom)is a matrix with $0$ and $1$ entries, defined in such a way that $\vec z = C\vec r$ is obtained by  ``copying'' each entry of the vector $\vec r$ (which corresponds to a node $y_i$, $i \in \mathcal{Y}\setminus\mathcal{X}$ and to a subdomain $\ell$ with $y_i \in \Yl$) in the line of $\vec z$ corresponding to the same node and subdomain.

	Then the for the  FETI-DP  method, the matrix counterpart of \nref{system_FETI}  and \nref{prec_FETI},  are 
	\begin{equation}\label{system_FETI_alg}
		B  \tSmat^{-1} B^T  \lambda  =  {\vec g}, \qquad \qquad   M_{FETI} = B_D \tSmat B_D^T,
	\end{equation}
	with ${\vec g}$  proper right-hand side.

	\section{Numerical results.}\label{sec:expes}
	In this section we present numerical results  for the model elliptic problem
	\[
	- \nabla\cdot (\rho \nabla u) = f \ \text{ in } \Omega = ]0,1[^2, \qquad u = 0 \ \text{ on } \partial \Omega.
	\]
	
	We consider both a constant coefficient
	problem  with  $\rho=1$ and several problems with 
	highly varying coefficients, with jumps across the interface. The domain $\Omega$ is decomposed into  $N^2$ square subdomains as shown in Figure \ref{fig:subdomains_ini_grids}. Two types of discretizations are considered: a structured discretization 
	with hexagonal elements and a Voronoi discretization (Figure \ref{fig:subdomains_ini_grids}, \ref{fig:subdomains_ini_grids_Voro}).

	The interface problem~\eqref{system_FETI_alg} is solved by a preconditioned conjugate gradient method (PCG) with zero initial guess and relative tolerance of $10^{-6}$. The condition numbers are numerical approximations computed from the standard tridiagonal Lanczos matrix generated during the PCG iteration as the ratio between the maximum and the minimum eigenvalues, see e.g.  \cite{GolubVanLoan_1996}.
	All the numerical tests are performed in {\sc matlab R2014}b$^\copyright$.

	\subsection{The constant coefficient case}
	We start by testing the scalability and optimality of the methods on the simplest case of constant $\rho = 1$. The right hand side is chosen as $f=\sin(\pi x) \sin(\pi y)$. 
	Inside each subdomain, we consider virtual elements of order $k = 1$. 
	In Tables \ref{ex_ro1_Nfix_c} and  \ref{ex_ro1_Nfix_c_Vor} (referring resp. to hexagonal and Voronoi meshes) we report the iteration counts and condition numbers of the FETI-DP preconditioned system \nref{system_FETI_alg} 
	for increasing $N$ (number of subdomains in each coordinate direction) and $n$ (number of elements in each subdomain). 
	The same quantities for the BDDC preconditioned system \nref{system_BDDC_alg}-\nref{prec_BDDC_alg} on hexagonal meshes
	are reported in Table  \ref{ex_ro1_Nfix_c_BDDC}.  
	In all three tests the largest problem ($N=64$, $n = 70 \times 80$) could not be tested due to memory limitations. 
	The results are in agreement with the theoretical bound on the condition numbers (see Corollary~\nref{cond_bound}), moreover,
	comparing Tables  \ref{ex_ro1_Nfix_c}   and  \ref{ex_ro1_Nfix_c_BDDC}, 
	it clearly appears that the two spectra are almost identical. 
	The small differences  can be attributed to the different systems being solved: 
	the Schur complement system
	\nref{system_BDDC_alg}-\nref{prec_BDDC_alg}
	for BDDC and the Lagrange multipliers system
	\nref{system_FETI_alg}
	for FETI-DP; the different convergence histories for the two systems lead in some cases to slightly different iteration counts and Lanczos approximations of the extreme eigenvalues.

	\begin{figure}[htbp]
		\begin{center}
			\includegraphics[width=0.45\textwidth]{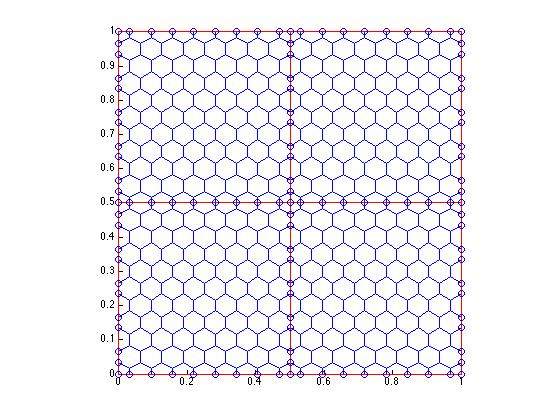}
			\includegraphics[width=0.45\textwidth]{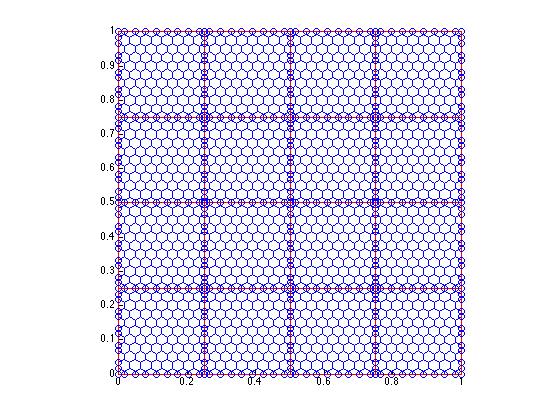}
			\caption{Initial grid of hexagons on a subdomain partition made by $2\times 2$ squares (left) and $4\times 4$. }
			\label{fig:subdomains_ini_grids}
		\end{center}
	\end{figure}
	
	\begin{figure}[htbp]
		\begin{center}
			\includegraphics[height=0.25\textwidth,width=0.25\textwidth]{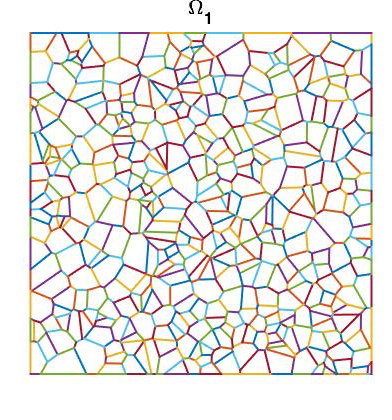}\hspace{.5cm}
			\includegraphics[width=0.55\textwidth]{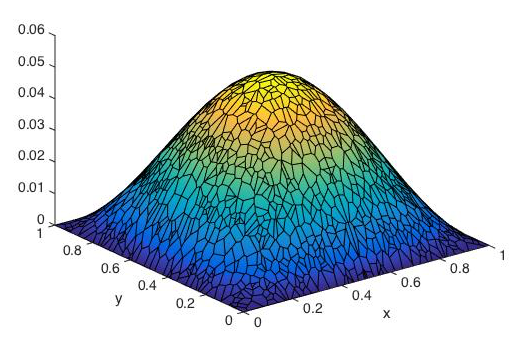}
			\caption{Voronoi mesh in each subdomain. {\em Left.} Mesh in $\Omega_1$. {\em Right} Computed solution with FETI-DP preconditioner. } 
			\label{fig:subdomains_ini_grids_Voro}
		\end{center}
	\end{figure}

	\begin{table}[htbp]
		\caption{Hexagonal mesh. Number of iterations and condition number of the FETI-DP preconditioned system preconditioners for varying number of subdomains  in each coordinate direction $N$ and local problem size $n$. } 
		\begin{center}
			\begin{tabular}{|c||c|c |  c | c | }
				\hline
				$N\backslash n$ & $8\times 10$ & $18\times 20$ & $34\times 40$ & $70\times 80$ \\
				\hline
				$8$  & 9 (3.61) & 11 (4.81) & 11 (5.86) & 12 (7.14) \\
				$16$ & 9 (3.71) & 11 (4.92) & 12 (5.99) & 14 (7.40) \\
				$32$ & 9 (3.75) & 10 (4.95) & 11 (6.02) & 13 (7.49)\\
				$64$ & 8 (3.76) &  9 (4.96) & 10 (6.03) & -- \\
				\hline
			\end{tabular}
		\end{center}
		\label{ex_ro1_Nfix_c}
	\end{table}%

	\begin{table}[htbp]
		\caption{Hexagonal mesh. Number of iterations and condition number of the BDDC preconditioned system preconditioners for varying number of subdomains in each coordinate direction $N$ and  local problem size $n$.}
		\begin{center}
			\begin{tabular}{|c||c|c |  c | c | }
				\hline
				$N\backslash n$ & $8\times 10$ & $18\times 20$ & $34\times 40$ & $70\times 80$ \\
				\hline
				$8$  & 10 (3.64) & 11 (4.81) & 11 (5.86) & 12 (7.14) \\
				$16$ & 9 (3.71) & 11 (4.92) & 12 (5.99) & 14 (7.41) \\
				$32$ & 9 (3.75) & 10 (4.95) & 12 (6.03) & 13 (7.49)\\
				$64$ & 8 (3.77) &  10 (4.97) & 10 (6.03) & -- \\
				\hline
			\end{tabular}
		\end{center}
		\label{ex_ro1_Nfix_c_BDDC}
	\end{table}%
	
	\begin{table}[htbp]
		\caption{Voronoi mesh. Number of iterations and condition number of the FETI-DP preconditioned system preconditioners for varying number of subdomains in each coordinate direction $N$ and  local problem size $n$.} 
		\begin{center}
			\begin{tabular}{|c||c|c | c | c | }
				\hline
				$N\backslash n$ & $100$ & $400$ & $1400$ & $5000$ \\
				\hline
				$8$  & 9 (2.88) & 11 (3.98) & 11 (4.90)  & 12 (5.67) \\
				$16$ & 9 (2.94) &12 (4.06)  & 12 (5.02) & 12 (5.79) \\
				$32$ & 9 (2.95) & 11 (4.07) & 11 (5.05)  & 12 (5.82)  \\
				$64$ & 9 (2.96) & 11 (4.08) & 11 (5.06)   & - \\
				\hline
			\end{tabular}
		\end{center}
		\label{ex_ro1_Nfix_c_Vor}
	\end{table}%

	\subsection{ Discontinuous coefficients $\rho$ across the skeleton.}
	In order to verify the robustness of the method with respect to strong variations of the coefficient $\rho$, we next consider problems with $\rho$ discontinuous  across the interfaces $\Gamma$. 
	Inside each subdomain, we consider virtual elements of order $k = 1$. 
	Since  the BDDC and FETI-DP spectra coincide, we only report  the results for FETI-DP. We perform two set of tests.
	In Test A, the coefficient $\rho$ is equal to 1 except in a square at the center of $\Omega$, where it assumes  a constant value $\rho_0$ (see  Figure \ref{fig:jump_ro}-Left). Tests are performed for $\rho_0 = 10^{-4},\, 10^{-2},\, 1, 10^{2}, 10^{4}$.
	In Test B, $\rho$ takes the value $\rho= 10^{\alpha}$, 
	where, for each subdomain, the value of the exponent $\alpha$ is a random integer belonging to $[-4,\dots,4]$ (see Figure \ref{fig:jump_ro}-Right).
	{In order to test the robustness of the method,  for both test cases we chose, 
		as  load term $\hat{\vec f}$ 
		a uniformly distributed random vector.}

	\begin{figure}[h!]
		\begin{center}
			\includegraphics[width=0.45\textwidth]{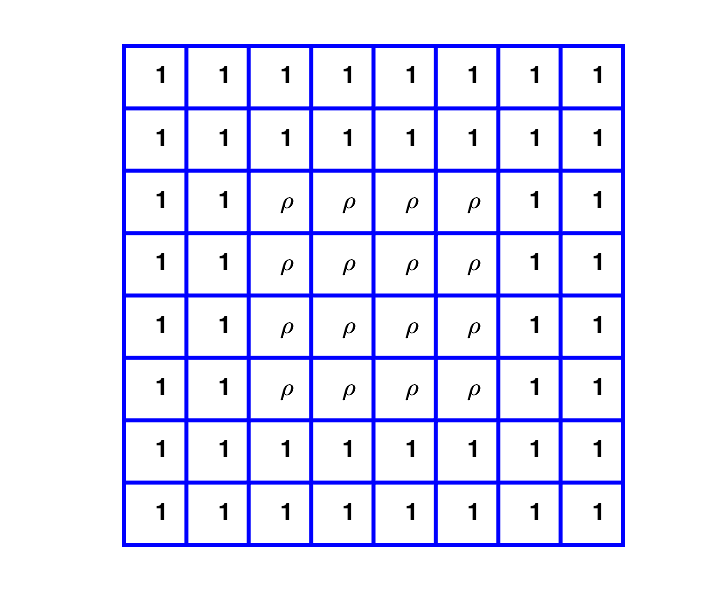} 
			\includegraphics[width=0.45\textwidth]{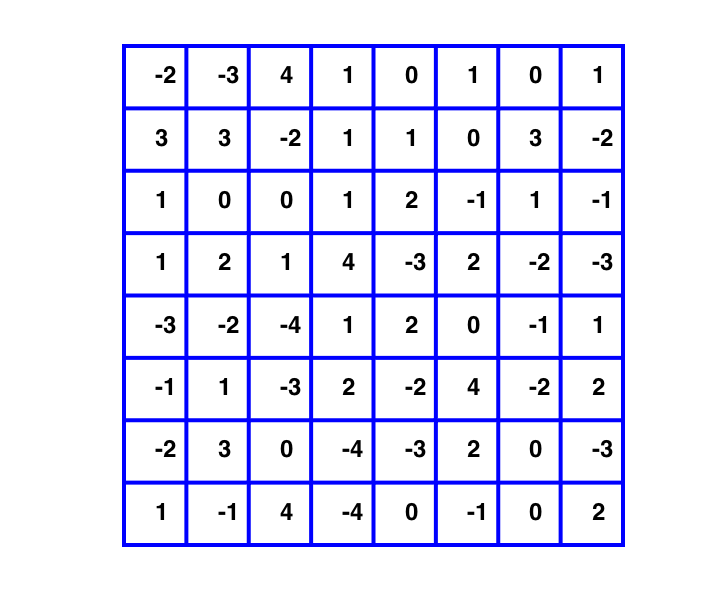} 
			\caption{{\em Left :} values of the discontinuous coefficient $\rho$ in Test A.  {\em Right}:  random distribution of the exponent $\alpha\in[-4,\dots,4]$ for the elliptic coefficient $\rho=10^\alpha$ in Test B. 
			}
			
			\label{fig:jump_ro}
		\end{center}
	\end{figure}
	
	\begin{table}[h!]
		\caption{Number of iterations and extreme eigenvalues of unpreconditioned CG and the FETI-DP preconditioned systems for two tests problems (A and B) with discontinuous coefficients $\rho_i$ shown in Figure \ref{fig:jump_ro}.}
		\begin{center}
			\begin{tabular}{|c c |c l l |c l l |}
				\hline
				Test & $\rho$ & \multicolumn{3}{l}{Unpreconditioned}    & FETI-DP & &  \\
				& & it. & $\lambda_{min}$ & $\lambda_{max}$ & it. & $\lambda_{min}$ & $\lambda_{max}$\\
				\hline
				A & $10^{-4}$  & 145  & 2.54     & 3.104e5  &  9   & 1.01 & 3.57 \\
				& $10^{-2}$ & 124  & 1.36     & 3.104e3  &  9   & 1.01 & 3.58 \\
				& $1$       &   25 & 1.34     & 32.58    &  9   & 1.02 & 3.67 \\
				& $10^{2}$  & 110  & 1.43e-2  & 27.61    &  9   & 1.02 & 3.58\\
				& $10^{4}$  & 109  & 3.83e-4  & 27.60    &  9   & 1.02 & 3.57 \\
				\hline
				B & $10^{\alpha}$ & $ > 1000$       &- & -  &  9   & 1.00 & 3.70  \\
				\hline
			\end{tabular}
		\end{center}
		\label{ex_ro_AB}
	\end{table}%
	
	In Table \ref{ex_ro_AB} we present the results for Test A and B for a fixed decomposition made of 
	$N=8$ subdomains in each coordinate direction and $n=8\times 10 $ hexagons in each subdomain $\Omega_\ell$. The results confirm that the convergence rate of FETI-DP  is quite insensitive to the coefficient jumps, with iteration counts equal to 9 and minimum eigenvalues very close to 1. In contrast, CG without FETI-DP preconditioner is  significantly affected by the  
	increasing jumps:  in the first test the condition number is of order $10^5$ for both $\alpha=10^{-4}$ and $\alpha = 10^4$, and CG does not converge within 1000 iterations in Test B.

	\begin{table}[h!]
		\caption{Hexagonal mesh. Number of iterations and condition numbers of the FETI-DP preconditioned system for varying number of subdomains in each coordinate direction $N$ and  local problem size $n$. Random distribution of the exponent $\alpha\in[-4,\dots,4]$ in the elliptic coefficient $\rho=10^\alpha$.}
		\begin{center}
			\begin{tabular}{|c||c|c |  c | c | }
				\hline
				$N\backslash n$ & $8\times 10$ & $18\times 20$ & $34\times 40$ & $70\times 80$ \\
				\hline
				$8$  & 10 (3.27) & 11 (4.16) & 13 (5.14) &  15 (6.58) \\
				$16$ & 11 (3.28) & 13 (4.21) & 15 (5.28) & 16 (6.79) \\
				$32$ & 11 (3.34) & 13 (4.36) & 15 (5.25) & 16 (6.73) \\
				\hline
			\end{tabular}
		\end{center}
		\label{ex_ro_randexp_Nfix_c}
	\end{table}%

	\begin{table}[h!]
		\caption{Voronoi mesh. Number of iterations and condition numbers of the FETI-DP preconditioned system for varying number of subdomains in each coordinate direction $N$ and  local problem size $n$. Random distribution of the exponent $\alpha\in[-4,\dots,4]$ in the elliptic coefficient $\rho=10^\alpha$.}
		\begin{center}
			\begin{tabular}{|c||c|c |  c | c | }
				\hline
				$N\backslash n$ & $100$ & $400$ & $1400$ & $5000$ \\
				\hline
				$8$  &   9 (3.07) & 11 (4.10) & 12 (4.96) & 12 (5.73) \\
				$16$ & 10 (3.08) & 12 (4.19) & 13 (4.89) & 15 (6.00) \\
				$32$ & 10 (3.00) & 12 (4.24) & 13 (5.01) & 15 (5.85) \\
				\hline
			\end{tabular}
		\end{center}
		\label{ex_ro_randexp_vor_Nfix_c}
	\end{table}%

	In Table \ref{ex_ro_randexp_Nfix_c} and \ref{ex_ro_randexp_vor_Nfix_c} we report  the iteration counts and condition numbers of the FETI-DP preconditioned system \eqref{system_FETI_alg} for  increasing $N$ (number of subdomains in each coordinate direction) and $n$ (number  of elements per subdomain) for Test B.
	Tables \ref{ex_ro_randexp_Nfix_c} and  \ref{ex_ro_randexp_vor_Nfix_c} refer to the case of hexagonal and Voronoi meshes, respectively. The results are consistent with the theoretical bound on the condition numbers (see Corollary~\nref{cond_bound}), i.e. we have a polylogarithmic dependence of the convergence rate on $H/h$ not affected by the jump in the coefficients $\rho$.

	\subsection{High-order elements.}
	Finally, we present some computations performed with high-order elements. 
	We take $\rho = 1$ and $f = \sin(\pi x) \sin(\pi y)$.
	Table  \ref{ex_ro1_Nfix_c_p2}  shows the number of iterations and the
	maximum eigenvalue $\lambda_{\max}$ for the FETI-DP preconditioned system, for fixed local mesh size ($n = 8 \times 10$ hexagons) 
	but varying both the polynomial degree $k$ from 2 to 8, and the number of subdomains in each coordinate direction $N$ from $8$ to $32$.
	The minimum eigenvalue is not reported since it is always very close to 1.  
	The BDDC results are analogous and therefore not displayed.
	
	In Figure \ref{p_cond}-Left, it is possible to verify the polylogarithmic dependence of the condition number on the polynomial order $k$,
	as predicted by the theoretical bound \nref{cond_bound}.
	Indeed, for $H\slash h$ fixed, 
	the expected bound $(1+\log(k^2 H\slash h)^2 \sim (1 +log(k^2))^2 \sim(1 + 2\log(k))^2 $
	is a linear semilog plot of the square root of $\lambda_{\max}$.
	In Figure \ref{p_cond}-Right, we keep the polynomial degree fixed to $k=4$ and $k=8$ and increase $H\slash h$, that is  
	the number of interior elements in each subdomain is increased from $8\times 10$ to $18\times 20$. 
	The largest eigenvalue and thus the condition number behaves as expected and the $(1+log(k^2 H\slash h))^2 \sim(1+log( H\slash h))^2$ bound (for fixed $k$) is confirmed.

	\begin{table}[htbp]
		\caption{
			Number of iterations and 
			maximum eigenvalues of the FETI-DP preconditioned system for fixed  local problem size ($n=8\times10$ hexagons) but increasing  polynomial order $k$ and subdomain partitions in each coordinate direction $N$.}
		\label{ex_ro1_Nfix_c_p2}
		\begin{center}
			\begin{tabular}{|c||c|c |  c | c | c | c | c | }
				\hline
				$N\backslash k$ & 2 & 3 & 4 & 5 & 6 & 7 & 8 \\
				\hline
				$8$  & 11 (5.72) & 12 (7.51) & 13 (8.56) & 13 (9.62)   & 14 (10.74) & 14 (11.50) & 14 (12.19) \\
				$16$ & 12 (5.85) & 14 (7.67) & 14 (8.72) & 15 (9.90)   & 16 (10.86) & 17 (11.74) & 17 (12.44)\\
				$32$ & 11 (5.88) & 13 (7.70) & 14 (8.90) & 14 (10.01) & 14 (10.96) & 16 (11.81) &  -  \\
				\hline
			\end{tabular}
		\end{center}
	\end{table}%

	\begin{figure}[htbp]
		\begin{center}
			\includegraphics[width=0.45\textwidth]{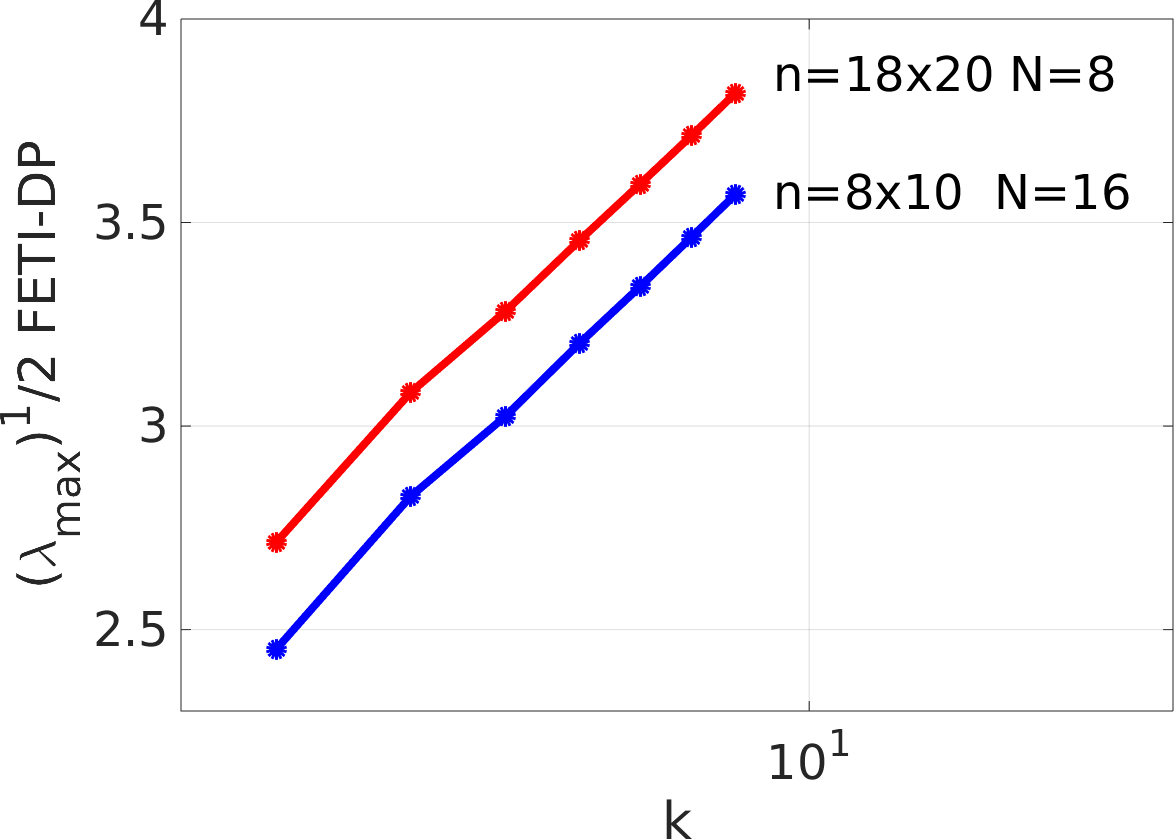}  \quad
			\includegraphics[width=0.45\textwidth]{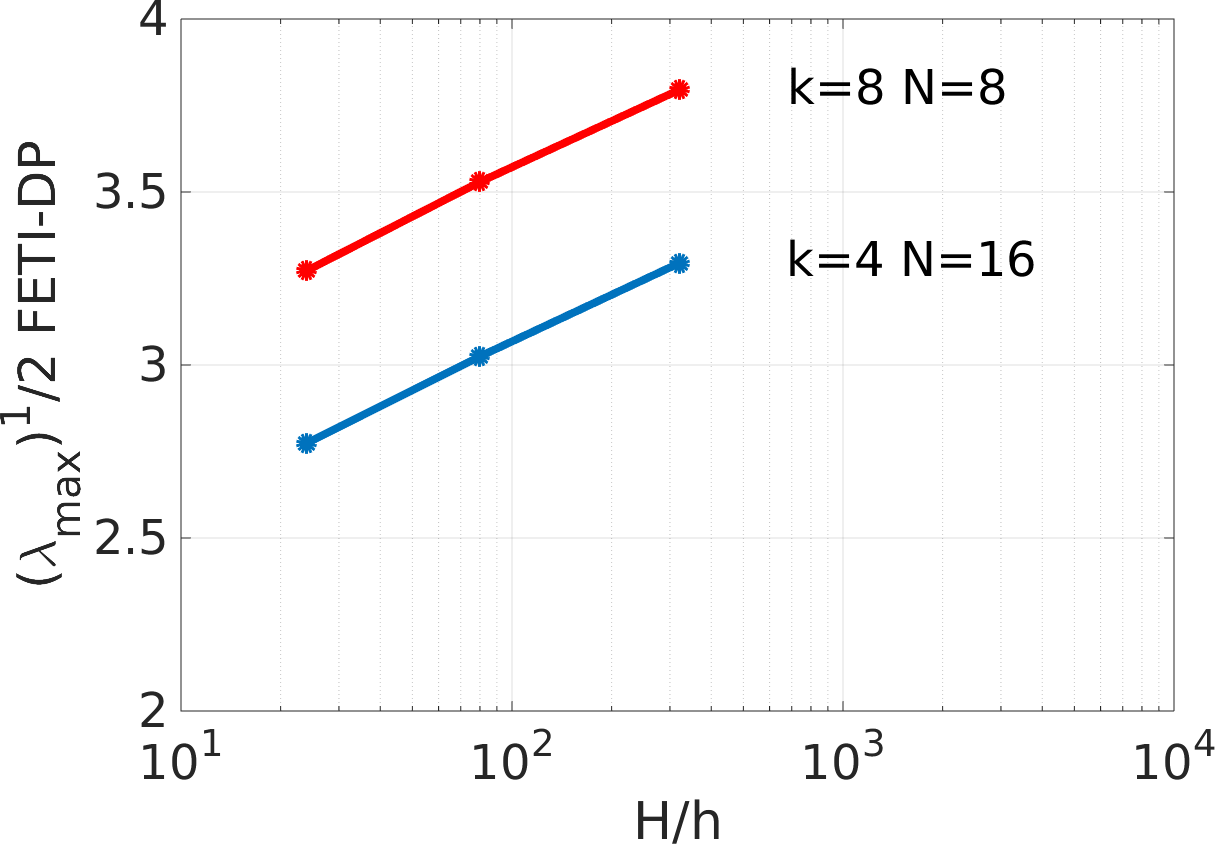}    
			\caption{{\em Left.} Square root of the largest eigenvalue $\sqrt{\lambda_{\max}}$ for increasing polynomial degree, 
				$n=18\times20$, $N=8$ and $n=8\times10$, $N=16$. 
				{\em Right.} Square root of the largest eigenvalue $\sqrt{\lambda_{\max}}$ for increasing  $H\slash h$ and $k=4,8$, $N=16,8$.}
			\label{p_cond}
		\end{center}
	\end{figure}
	
	



\end{document}